\documentclass[10pt]{amsart}
\title{The point variety of quantum polynomial rings}
\author{Pieter Belmans}
\address{Department of Mathematics, University of Antwerp \\ 
 Middelheimlaan 1, B-2020 Antwerp (Belgium) \\ {\tt pieter.belmans@uantwerpen.be} \\ supported by a Ph.D.\ fellowship of the Research Foundation---Flanders (FWO)}
\author{Kevin De Laet}
\address{Department of Mathematics, University of Antwerp \\ 
 Middelheimlaan 1, B-2020 Antwerp (Belgium) \\ {\tt kevin.delaet2@uantwerpen.be}}
\author{Lieven Le Bruyn}
\address{Department of Mathematics, University of Antwerp \\ 
 Middelheimlaan 1, B-2020 Antwerp (Belgium) \\ {\tt lieven.lebruyn@uantwerpen.be}}

\date{}
\usepackage{etex}
\usepackage{amsmath,amsfonts,array,amscd,amsthm,amssymb,stackrel,verbatim,wrapfig,subfig,float}
\usepackage{enumerate}
\usepackage{booktabs}
\usepackage{url}
\usepackage{tikz}
\usepackage[all,cmtip]{xy}
\usetikzlibrary{arrows,matrix}
\usetikzlibrary{shapes.geometric}
\usetikzlibrary{decorations.markings} 

\tikzset{
  vertice/.style={circle,draw=black},
  decoration={markings,mark=at position 0.5 with {\arrow{>}}}
}

\newcommand{\wis}[1]{{\text{\em \usefont{OT1}{cmtt}{m}{n} #1}}}
\newcommand{\C}{\mathbb{C}}

\newcommand{\PP}{\mathbb{P}}

\theoremstyle{plain}
\newtheorem{theorem}{Theorem}
\newtheorem{lemma}[theorem]{Lemma}
\newtheorem{proposition}[theorem]{Proposition}

\newtheorem{remark}[theorem]{Remark}
\newtheorem{example}[theorem]{Example}

\DeclareMathOperator{\rk}{rk}
\DeclareMathOperator{\Sym}{S}

\numberwithin{equation}{section}
\begin{document}

\sloppy

\begin{abstract}
We show that the reduced point variety of a quantum polynomial algebra is the union of specific linear subspaces in $\mathbb{P}^n$, we describe its irreducible components and give a combinatorial description of the possible configurations in small dimensions.
\end{abstract}
\maketitle
\vskip 5mm

\section{Introduction}

Recall that a {\em quantum polynomial algebra} on $n+1$ variables has a presentation
\[
A = \C \langle x_0, x_1, \hdots, x_n \rangle / (x_i x_j - q_{ij} x_j x_i,~0 \leq i,j \leq n) \]
where all entries of the $n+1 \times n+1$ matrix $Q=(q_{ij})_{i,j}$ are non-zero and satisfy the relations $q_{ii}=1$ and $q_{ji}=q_{ij}^{-1}$.

If all the variables $x_i$ are given degree one, $A$ is a positively graded algebra with excellent homological conditions: it is an iterated Ore-extension and an Auslander-regular algebra of dimension $n+1$. In non-commutative projective geometry, see for example \cite{ATV1} or \cite{Stafford}, one associates to such algebras a {\em quantum projective space} defined by
\[
\mathbb{P}^n_Q = \mathbf{Proj}(A) = \mathbf{Gr}(A)/\mathbf{Tors}(A) \]
where $\mathbf{Proj}(A)$ is the quotient category of the category $\mathbf{Gr}(A)$ of all graded left $A$-modules by the Serre subcategory $\mathbf{Tors}(A)$ of all graded torsion left $A$-modules.

An interesting class of objects in $\mathbb{P}^n_Q$ are the {\em point modules} of $A$, which are determined by graded left $A$-modules $P=P_0 \oplus P_1 \oplus \hdots$ which are {\em cyclic} (that is, are generated by one element in degree zero), {\em critical} (implying that all normalizing elements of $A$ act on it either as zero or as a non-zero divisor) and have Hilbert-series $(1-t)^{-1}$ (that is all graded components $P_i$ have dimension one).
As such a point module can be written as a quotient $P \simeq A/(Al_1 + \hdots + Al_n)$ with linearly independent $l_i \in A_1$, we can associate to it a unique point $x_P = \mathbb{V}(l_1,\hdots,l_n)$ in commutative projective $n$-space $\mathbb{P}^n = \mathbb{P}(A_1^*)$, having as its projective coordinates $[u_0:u_1:\hdots : u_n]$ with $u_i=x_i^*$. The {\em point variety} of $A$ is then the reduced closed subvariety of $\mathbb{P}^n$
\[
\wis{pts}(A) = \{ x_p \in \mathbb{P}^n\mid\text{$P$ a point module of $A$} \} \]
The aim of this paper is to describe the possible subvarieties that can arise as point varieties of quantum polynomial algebras. We will prove the next result in Section~2.

\begin{theorem} With notations as above we have
\begin{enumerate}
\item{$\wis{pts}(A) = \mathbb{V}((q_{ij}q_{jk}-q_{ik})u_iu_ju_k,~0 \leq i < j < k \leq n)$ and hence is the union of a collection of linear subspaces of the form $\mathbb{P}(i_0,\hdots,i_k)$ which is the $k$-linear subspace of $\mathbb{P}^n$ spanned by $\delta_{i_0},\hdots,\delta_{ik}$ where $\delta_j = [\delta_{0j}:\hdots:\delta_{nj}]$.}
\item{$\mathbb{P}(i_0,\hdots,i_k)$ is an irreducible component of $\wis{pts}(A)$ if and only if the principal $k+1 \times k+1$ minor of $Q$
\[
Q(i_0,\hdots,i_k) = \begin{bmatrix} 1 & q_{i_0i_1} & \hdots & q_{i_0i_k} \\
q_{i_1 i_0} & 1 & \hdots & q_{i_1 i_k} \\
\vdots & \vdots & \ddots & \vdots \\
q_{i_k i_0} & q_{i_k i_1} & \hdots & 1 \end{bmatrix} \]
is maximal among principal $Q$-minors such that $\rk Q(i_0,\hdots,i_k)=1$.}
\item{$\wis{pts}(A) = \mathbb{V}(u_iu_ju_k;~0 \leq i < j < k \leq n, \mathbb{P}(i,j,k) \not\subset \wis{pts}(A))$.
In particular, the point variety of $A$ is determined by the $\mathbb{P}^2=\mathbb{P}(u,v,w)$ it contains.}
\end{enumerate}
\end{theorem}

In Section~3 we will give a necessary condition for a union of linear subspaces in $\mathbb{P}^n$ to be the point variety of a quantum polynomial algebra. Theorem~\ref{th:densadequate} implies that this condition is also sufficient for $n \leq 5$.

In Section~4 we list all possible configurations, and the corresponding degeneration graph, when $n \leq 4$. In dimension $5$ the degeneration graph no longer has a unique end-point.

\section{The proof}

Because each variable $x_i$ is a normalizing element in $A$ we can consider the graded localization at the homogeneous Ore set $\{ 1, x_i, x_i^2, \hdots \}$. As this localization has an invertible element of degree one it is a {\em strongly graded ring}, see \cite[\S 1.4]{NastaFVO}, and therefore is a skew Laurent extension
\[
A[x_i^{-1}] = B_i[x_i,x_i^{-1},\sigma] \]
where $B_i$ is the degree zero part of $A[x_i^{-1}]$ and where $\sigma$ is the automorphism on $B_i$ given by conjugation with $x_i$.

The algebra $B_i$ is generated by the $n$ elements $v_j = x_j x_i^{-1}$ and as we have the commutation relations $x_j x_i^{-1} = q_{ij} x_i^{-1} x_j$ we get the commutation relations
\[
  \begin{aligned}
    v_j v_k
    &= q_{ij} x_i^{-1} x_j x_k x_i^{-1} \\
    &= q_{ij}q_{jk} x_i^{-1} x_k x_j x_i^{-1} \\
    &= q_{ij} q_{jk} q_{ik}^{-1} x_l x_i^{-1} x_j x_i^{-1} \\
    &= q_{ij} q_{jk} q_{ik}^{-1} v_k v_j
  \end{aligned}
\]
That is, $B_i$ is again a quantum polynomial algebra, this time on $n$ variables $v_j$ with corresponding $n \times n$ matrix $R = (r_{jk})_{j,k}$ with entries
\[
r_{jk} = q_{ij} q_{jk} q_{ik}^{-1} \]
One-dimensional representations of $B_i$ correspond to points $(a_j)_j \in \mathbb{A}^n$ (via the morphism $v_j \mapsto a_j$) if they satisfy all the defining relations $v_j v_k = r_{jk} v_k v_j$ of $B_i$, that is,
\begin{equation}
(a_j)_j \in \bigcap_{j \not= i \not= k} \mathbb{V}((1-r_{jk})v_j^*v_k^*)  \end{equation}
Observe that we can identify this affine space $\mathbb{A}^n$ with $\mathbb{X}(u_i)$ in $\mathbb{P}^n$ with affine coordinates $v_j^* = u_j u_i^{-1}$. That is, we can identify the projective closure of $\wis{rep}_1(B_i)$, the affine variety of all one-dimensional representations of $B_i$, with the following subvariety of $\mathbb{P}^n$
\[
\overline{\wis{rep}_1(B_i)} = \bigcap_{j \not= i \not= k} \mathbb{V}((q_{ik}-q_{ij}q_{jk}) u_j u_k). \]
\vskip 3mm

\noindent
\begin{proof}[Proof of theorem 1.(1)]
Let $A = \C \langle x_0 , x_1 ,x_2 \rangle /(x_i x_j - q_{ij} x_j x_i, 0\leq i,j \leq 2)$ be a quantum polynomial algebra in 3 variables.
Then $\wis{pts}(A)$ is determined (see \cite{ATV2}) by the determinant of the following matrix
$$
\begin{bmatrix}
- q_{01} u_1 & u_0 & 0 \\ 0 & - q_{12} u_2 & u_1 \\  - q_{02} u_2 & 0 & u_0
\end{bmatrix}
$$
which is equal to $(q_{01} q_{12} - q_{02})u_0 u_1 u_2$. This proves the claim for $n=2$.
\par Let $A$ now be a quantum polynomial algebra in $n+1$ variables. If $P$ is a point module of $A$, then each of the variables $x_i$ (being normalizing elements) either acts as zero on $P$ or as a non-zero divisor. At least one of the $x_i$ must act as a non-zero divisor (otherwise $P \simeq \mathbb{C} =A/(x_0,\hdots,x_n)$), but then the localization $P[x_i^{-1}]$ is a graded module over the strongly graded ring $B_i[x_i,x_i^{-1},\sigma]$ and hence is fully determined by its part of degree zero $(P[x_i^{-1}])_0$, see \cite[\S 1.3]{NastaFVO} or \cite[Proposition 7.5]{ATV1}, which is a one-dimensional representation of $B_i$ and so $P$ determines a unique point of $\wis{rep}_1(B_i)$ described above. Hence, we have the decomposition
\begin{equation}
\wis{pts}(A) = \wis{rep}_1(B_i) \sqcup \wis{pts}(A/(x_i))). 
\end{equation}
$A/(x_i)$ is a quantum polynomial algebra in $n$ variables. Hence by induction, we have $$\wis{pts}(A/(x_i)) = \bigcap_{j \neq i, k\neq i, l\neq i} \mathbb{V}((q_{jl}-q_{jk}q_{kl}) u_j u_k u_l) \cap \mathbb{V}(u_i).$$
But then we have
\begin{align*}
\wis{pts}(A) &= \overline{\wis{rep}_1(B_i)}\cup \wis{pts}(A/(x_i))) \\
                 &= \bigcap_{j \not= i \not= k} \mathbb{V}((q_{ik}-q_{ij}q_{jk}) u_j u_k) \cup \bigcap_{j \neq i, k\neq i, l\neq i} \mathbb{V}((q_{jl}-q_{jk}q_{kl}) u_j u_k u_l) \cap \mathbb{V}(u_i)\\
                 &= \bigcap_{0\leq i < j < k \leq n} \mathbb{V}((q_{ik}-q_{ij}q_{jk}) u_i u_j u_k)
\end{align*}
The last equality follows from the following lemma.
\begin{lemma}
Fix $0\leq j<k<l \leq n$. If there exists an $i$ such that
\[
\begin{cases}
q_{ik}-q_{ij}q_{jk} = 0,\\
q_{il}-q_{ij}q_{jl} = 0,\\
q_{il}-q_{ik}q_{kl} = 0,
\end{cases}
\]
then $q_{jl}-q_{jk}q_{kl}=0$.
\end{lemma}
\begin{proof}
Easy calculation.
\end{proof}
From the lemma it follows that if $u_j u_k u_l$ belongs to the defining ideal of $\wis{pts}(A/(x_i))$, then necessarily for each $i$ either $u_j u_k$, $u_j u_l$ or $u_k u_l$ belongs to the defining ideal of $\overline{\wis{rep}_1(B_i)}$.
\end{proof}
\vskip 5mm

In particular, it follows that $\wis{pts}(A) = \mathbb{P}^n$ if and only if for all $j,k \not= i$ we have the relation
\[
q_{jk} = q_{ik} q_{ij}^{-1} \]
But then, all $2 \times 2$ minors of $Q$ have determinant zero as
\[
  \begin{bmatrix} q_{ju} & q_{jv} \\ q_{lu} & q_{lv} \end{bmatrix} = \begin{bmatrix} q_{iu}q_{ij}^{-1} & q_{iv}q_{ij}^{-1} \\
  q_{iu}q_{il}^{-1} & q_{iv}q_{il}^{-1} \end{bmatrix} \]
and the same applies for $2 \times 2$ minors involving the $i$-th row or column, so $Q$ must have rank one.

\vskip 3mm

\begin{proof}[Proof of theorem 1.(2)]
Observe that $\mathbb{P}(i_0,\hdots,i_k) = \mathbb{V}(j_1,\hdots,j_{n-k})$ where $\{ 0,1,\hdots,n \} = \{ i_0,\hdots,i_k \} \sqcup \{ j_1,\hdots,j_{n-k} \}$. Therefore, $\mathbb{P}(i_0,\hdots,i_k) \subset \wis{pts}(A)$ if and only if
\[
\mathbb{P}(i_0,\hdots,i_n) = \wis{pts}(\overline{A}) \quad \text{with} \quad \overline{A} = \frac{A}{(x_{j_1},\hdots,x_{j_{n-k}})} \]
and as $\overline{A}$ is again a quantum polynomial algebra with corresponding matrix $Q(i_0,\hdots,i_k)$ it follows from the remark above that $\rk Q(i_0,\hdots,i_k)=1$. 
\end{proof}

\vskip 3mm

\begin{proof}[Proof of theorem 1.(3)]
Recall that $\mathbb{P}(u,v,w) \subset \wis{pts}(A)$ if and only if $Q(u,v,w)$ has rank one, which is equivalent to $q_{uw}=q_{uv}q_{vw}$. The statement now follows from theorem 1.(1). 
\end{proof}

\begin{remark}
Observe that point varieties of quantum polynomial algebras always contain the $1$-skeleton of coordinate $\mathbb{P}^1$'s as the principal~$2\times 2$-minors always have rank~1. This will also be the generic configuration for quantum polynomial algebras. Note that in general noncommutative $\mathbb{P}^n$ can have no points or only a finite number of point modules, see \cite{vancliff} for examples when $n=3$.
\end{remark}

\section{Possible configurations}

Not all configurations of linear subspaces of the above type can occur as point varieties of quantum polynomial algebras.

\begin{example}
  \label{example:P3-two-P2s}
  In $\mathbb{P}^3$ only two of the $\mathbb{P}^2$'s (out of four in total) can arise in a proper subvariety $\wis{pts}(A) \subsetneq \mathbb{P}^3$. For example, take
\[
Q = \begin{bmatrix} 1 & a & b & x \\ a^{-1} & 1 & a^{-1} b & c \\ b^{-1} & ab^{-1} & 1 & a^{-1}bc \\ x^{-1} & c^{-1} & ab^{-1}c^{-1} & 1 \end{bmatrix} \]
then, for generic $a,b,c,x$ we have
\[
\wis{pts}(A) = \mathbb{P}(0,1,2) \cup \mathbb{P}(1,2,3) \cup \mathbb{P}(0,3) \]
However, if we include another $\mathbb{P}^2$, for example, $\mathbb{P}(0,1,3)$ we need the relation $x=ac$ in which case $Q$ becomes of rank one, whence $\wis{pts}(A) = \mathbb{P}^3$. This is a consequence of lemma~2.
\end{example}

We will present a combinatorial description of all possible configurations in low dimensions. Let $C$ be a collection of $\mathbb{P}^2=\mathbb{P}(i,j,k)$ contained in $\mathbb{P}^n$. We say that $C$ is {\em adequate} if the following condition is satisfied
\[
\forall~0 \leq i \leq n, \forall~\mathbb{P}(j,k,l) \in C, \exists~\{ u,v \} \subset \{ j,k,l \}~:~\mathbb{P}(i,u,v) \in C \]
Adequacy gives a necessary condition on the collection of $\mathbb{P}^2$'s not contained in the point variety of a quantum polynomial algebra.

\begin{proposition}\label{proposition:adequacy} If $A$ is a quantum polynomial algebra, then
\[
C_A = \{ \mathbb{P}(i,j,k)\mid \mathbb{P}(i,j,k) \not\subset \wis{pts}(A) \} \]
is an adequate collection.
\end{proposition}

\begin{proof}
It follows immediately from the description of $\wis{pts}(A) \cap \mathbb{X}(u_i)$ given by equation~(2.1) that $C_A$ is indeed adequate.
\end{proof}

The collection of all coordinates $(q_{ij})_{i<j}$ in the torus of dimension $\binom{n+1}{2}$ describing quantum polynomial algebras with the same reduced point variety is an open subset $T$ of a torus with complement certain sub-tori describing the coordinates of quantum algebras with larger point variety.

In example~\ref{example:P3-two-P2s} we have $C_A=\{ \mathbb{P}(0,1,3), \mathbb{P}(0,2,3) \}$ and $T$ is the complement of $(\C^*)^4$ (with coordinates $a,b,c,x$) by the sub-torus $(\C^*)^3$ defined by $x=ac$, describing quantum polynomial algebras with point variety $\mathbb{P}^3$. Here, $C_A$ is adequate, but for example $C=\{ \mathbb{P}(0,1,3) \}$ is not. In fact, for $n=3$ it is easy to check that all collections are adequate apart from the singletons, so there are exactly $12$ adequate collections.

We say that a collection $C$ of $\mathbb{P}^2$'s in $\mathbb{P}^n$ is {\em dense} if there exist $0 \leq i < j \leq n$ such that
\[
\#~\{ \mathbb{P}(i,j,k) \in C \} \geq n-2 \]
where~$k\neq i,j$. For small $n$, adequate collections are always dense.

\begin{proposition} \label{proposition:denseness-n4} For $n \leq 4$ all adequate collections are dense unless $C = \emptyset$.
\end{proposition}

\begin{proof} 
For $n=2$, the proof is trivial. For $n=3$, it is easily seen that that $C = \emptyset$ or $C = \{\mathbb{P}(i,j,k)\}$ are the only non-dense collections. It is trivial that $C = \{\mathbb{P}(i,j,k)\}$ is not an adequate collection.
\par Assume now that $n=4$ and that $C$ is a non-dense collections. Then we have for all $0 \leq i<j \leq 4$ that 
\[
\#~\{ \mathbb{P}(i,j,k) \in C \} = 0,1. \]
If this quantity is always equal to 0 then $C = \emptyset$, which is adequate. Hence, assume that one of these quantities is equal to 1. Up to permutation by $\Sym_5$, we may assume that $\mathbb{P}(0,1,2) \in C$. Then the only possible $\mathbb{P}(i,j,k)$ belonging to $C$ is $\mathbb{P}(i,3,4)$ with $i$ either $0,1$ or $2$. Again up to permutation, we may assume $i=0$. But neither the collection $\{\mathbb{P}(0,1,2)\}$ nor 
$\{\mathbb{P}(0,1,2), \{\mathbb{P}(0,3,4)\}\}$ are adequate (in both cases, take $i=3$ and $\mathbb{P}(0,1,2)$).
\end{proof}

We can now characterize the possible configurations in small dimensions.

\begin{theorem} Assume $n \leq 5$ and let $C$ be an adequate and dense collection of $\mathbb{P}^2$'s in $\mathbb{P}^n$ with variables $u_i$ for $0 \leq i \leq n$. Then, 
\[
\mathbb{V}(u_iu_ju_k\mid\mathbb{P}(i,j,k) \in C) \]
is the point variety $\wis{pts}(A)$ of a quantum polynomial algebra with $C=C_A$.
\label{th:densadequate}
\end{theorem}

\begin{proof} Renumbering the variables if necessary we may assume by denseness that $\mathbb{P}(0,n)$ is contained in at least $n-2$ of $\mathbb{P}(0,i,n) \in C$. We can write $C$ as a disjoint union $C_1 \sqcup C_2 \sqcup C_3 \sqcup C_4$ with
\[
\begin{cases}
C_1 = \{ \mathbb{P}(p,q,r) \in C\mid p,q,r \notin \{ 0,n \} \} \\
C_2 = \{ \mathbb{P}(0,p,q) \in C\mid p,q \not= 0,n \} \\
C_3 = \{ \mathbb{P}(p,q,n) \in C\mid p,q \not= 0,n \} \\
C_4 = \{ \mathbb{P}(0,p,n) \in C\mid p \notin \{ 0,n \} \} 
\end{cases}
\]
Note that $\#C_4 \geq n-2$. By adequacy of $C$ we have that $C_1$ is adequate in the variables $u_i$ for $1 \leq i \leq n-1$, $C_1 \sqcup C_2$ is adequate in the variables $u_i$ with $0 \leq i \leq n-1$ and $C_1 \sqcup C_3$ is adequate in the variables $u_i$ with $1 \leq i \leq n$.

Hence, by applying the induction hypothesis twice (which is possible by proposition~\ref{proposition:adequacy}), a first time with generic values for $C_1 \sqcup C_2$ and afterwards with specific values for $C_1 \sqcup C_3$, and evaluating the generic values accordingly, we obtain a matrix with non-zero entries
\[
Q = \begin{bmatrix} 1 & q_{01} & \hdots & q_{0 n-1} & x \\
q_{01}^{-1} & 1 & \hdots & q_{1 n-1} & q_{1 n} \\
\vdots & \vdots & \ddots & \vdots & \vdots \\
q_{0 n-1}^{-1} & q_{1 n -1}^{-1} & \hdots & 1 & q_{n-1 n} \\
x^{-1} & q_{1 n}^{-1} & \hdots & q_{n-1 n}^{-1} & 1 \end{bmatrix} \]
such that for all principal $3 \times 3$ minors $Q(i,j,k)$ with $\{ 0,n \} \not\subset \{ i,j,k \}$ we have
\[
\rk Q(i,j,k) = 1 \quad \text{if and only if} \quad \mathbb{P}(i,j,k) \notin C_1 \sqcup C_2 \sqcup C_3 \]
But then, the same condition is satisfied for all the matrices
\[
Q_{\lambda} = \begin{bmatrix} 1 & q_{01} & \hdots &  q_{0 n-1} & x \\
 q_{01}^{-1} & 1 & \hdots &  q_{1 n-1} & \lambda  q_{1 n} \\
\vdots & \vdots & \ddots & \vdots & \vdots \\
 q_{0 n-1}^{-1} &  q_{1 n -1}^{-1} & \hdots & 1 & \lambda q_{n-1 n} \\
x^{-1} & \lambda^{-1}  q_{1 n}^{-1} & \hdots & \lambda^{-1} q_{n-1 n}^{-1} & 1 \end{bmatrix} \]
with $\lambda \in \C^*$. 
If $\# C_4 = n-1$, a generic value of $x$ will ensure that all $\rk Q(0,j,n)> 1$ for $1 \leq j \leq n-1$.
If $\# C_4 = n-2$ let $i$ be the unique entry $1 \leq i \leq n-1$ such that $\mathbb{P}(0,i,n) \notin C$, then the rank one condition on
\[
Q(0,i,n) = \begin{bmatrix} 1 & q_{0i} & x \\
q_{0i}^{-1} & 1 & \lambda q_{in} \\
x^{-1} & \lambda^{-1} q_{in}^{-1} & 1 \end{bmatrix} \quad \text{implies} \quad \lambda = q_{0i}^{-1} q_{in}^{-1} x \]
and for generic $x$ we can assure that for all other $1 \leq j\not= i \leq n-1$ we have $\rk Q(0,j,n) > 1$.
\end{proof}

One can verify that, up to the $\Sym_6$-action on the variables $u_i$, there are exactly two adequate collections for $n=5$ which are {\em not} dense, which are:
\[
\mathcal{A}= \{ \mathbb{P}(0,2,4), \mathbb{P}(0,2,5), \mathbb{P}(0,3,4), \mathbb{P}(0,3,5), \mathbb{P}(1,2,4), \mathbb{P}(1,2,5), \mathbb{P}(1,3,4), \mathbb{P}(1,3,5) \} \]
 and
 \[
 \mathcal{B}=\{ \mathbb{P}(0,1,3), \mathbb{P}(0,1,5), \mathbb{P}(0,2,4), \mathbb{P}(0,4,5), \mathbb{P}(0,2,3), \mathbb{P}(1,2,4), \]
 \[
 \mathbb{P}(1,2,5), \mathbb{P}(1,3,4), \mathbb{P}(2,3,5), \mathbb{P}(3,4,5) \}. \]
$\mathcal{A}$ is realisable as $C_A$ for a quantum polynomial algebra $A$ with matrix 
 \[
 \begin{bmatrix}
 1 & 1 & 1 & 1 & x & x \\
 1 & 1 & 1 & 1 & x & x \\
 1 & 1 & 1 & 1 & 1 & 1 \\
 1 & 1 & 1 & 1 & 1 & 1 \\
 x^{-1} &  x^{-1} & 1 & 1 & 1 & 1 \\
 x^{-1} &  x^{-1} & 1 & 1 & 1 & 1 
\end{bmatrix}  \]
and has as point variety $\mathbb{P}(0,1,2,3) \cup \mathbb{P}(0,1,4,5) \cup \mathbb{P}(2,3,4,5)$ for generic $x$.
\par $\mathcal{B}$ is a $C_{A'}$ for the quantum algebra $A'$ with defining matrix 
\[
\begin{bmatrix}
  1  & -1 &  1 &  1 & -1 &  1\\
  -1 &  1 & -1 &  1 &  1 &  1\\
  1  & -1 &  1 & -1 &  1 &  1\\
  1  &  1 & -1 &  1 & -1 &  1\\
 -1  &  1 &  1 & -1 &  1 &  1\\ 
  1  &  1 &  1 &  1 &  1 &  1
\end{bmatrix}
\]
The point variety of this algebra is 
\[
\PP(0,1,2) \cup \PP(1,2,3) \cup \PP(2,3,4) \cup \PP(0,3,4) \cup \PP(0,1,4) \cup
\]
\[
\PP(0,2,5) \cup \PP(1,3,5) \cup \PP(2,4,5) \cup \PP(0,3,5) \cup \PP(1,4,5)
\]
\par This shows that denseness is {\em too} strong a condition for $C$ to be realised as $C_A$ for some quantum polynomial algebra $A$. However, these results may imply that adequacy is a sufficient condition. In particular, all~175 $\Sym_6$-equivalence classes of adequate collections in dimension~5 can be realised as the collection of $\mathbb{P}^2$'s not contained in the point variety of a quantum polynomial algebra on~6 variables.

\section{Degeneration graphs}
Let $\mathbb{T}_{2,n}$ be the $\binom{n+1}{2}$-dimensional torus parametrizing quantum polynomial algebras as before with coordinate functions $(q_{ij})_{i<j}$. Put $b_{ijk} = q_{ij}q_{jk}q_{ik}^{-1}$ for $0\leq i < j < k \leq n$ and let $I = \{b_{ijk}-1\mid 0\leq i < j < k \leq n\}$. For each $J \subset I$, we obtain a subtorus of $\mathbb{T}_{2,n}$ by taking $\mathbb{V}(J)$. Note however that $\mathbb{V}(J)$ can be equal to $\mathbb{V}(K)$ although $J \neq K$.
\par We obtain this way a degeneration graph by letting the nodes corresponds to possible $\mathbb{V}(J), J \subset I$ and an arrow $\mathbb{V}(J) \rightarrow \mathbb{V}(K)$ if $\mathbb{V}(K) \subset \mathbb{V}(J)$.
\par From the above description of point varieties of quantum polynomial algebras, we see that this degeneration graph corresponds to degenerations of quantum polynomial algebras to other quantum polynomial algebras with a larger point module variety.
\par Some considerations must be made in the calculations of these graphs:
\begin{itemize}
  \item Let $\mathbb{T}_{3,n}$ be the $\binom{n+1}{3}$-dimensional torus with coordinate functions $(b_{ijk})_{i<j<k}$. Then the map $\xymatrix{\mathbb{T}_{2,n}  \ar[r] & \mathbb{T}_{3,n}}$ defined by $b_{ijk} = q_{ij}q_{jk}q_{ik}^{-1}$ is a map of algebraic groups. The kernel $\mathbb{K}$ of this map is a $n$-dimensional torus which acts freely on each $\mathbb{V}(J)$ in the obvious way. Therefore, each $\mathbb{V}(J)$ is at least $n$-dimensional.
\item The nodes in our graphs are possible subtori up to $\Sym_{n+1}$-action on the variables of the quantum polynomial algebras.
\end{itemize}
For $n=2,3,4$, we have calculated the complete degeneration graphs using these methods. 

\subsection{Quantum $\mathbb{P}^2$'s}
This case is classical \cite{ATV1}: the point variety is either $\mathbb{P}^2$ or the union of the~3 coordinate~$\mathbb{P}^1$'s \cite{ATV2}.

\subsection{Quantum $\mathbb{P}^3$'s}
The degeneration graph is given in figure~\ref{figure:degeneration-P3}. One can easily check by hand that there are~12 adequate collections, that fall into~4~$\Sym_4$-orbits.

The label for a configuration corresponds to the dimension of the loci (in~$\mathbb{T}_{2,n}$) parametrising these configurations. The type of a configuration describes how many~$\mathbb{P}^k$'s there are as irreducible components in the point variety. The commutative situation where the point variety is the whole of~$\mathbb{P}^3$ therefore is labeled by~$0$ and has type~$(1,0,0)$, whereas the most generic situation (labeled by~$3$) corresponds to~6~$\mathbb{P}^1$'s whose type we denote by~$(0,0,6)$.

In this case the degeneration graph is totally ordered, with example~\ref{example:P3-two-P2s} corresponding to the configuration with label~$1$.

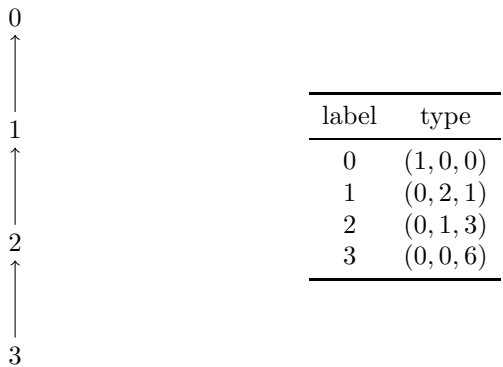
\begin{figure}[ht]
  \centering
  \begin{minipage}{0.4\textwidth}
    \centering
    \begin{tikzpicture}[scale = 1.5]
      \node (n0) at (0,3) {$0$};
      \node (n1) at (0,2) {$1$};
      \node (n2) at (0,1) {$2$};
      \node (n3) at (0,0) {$3$};
     
      \foreach \i/\j in {3/2,2/1,1/0}
        \draw[->] (n\i) -- (n\j);
    \end{tikzpicture}
  \end{minipage}
  \begin{minipage}{.4\textwidth}
    \centering
    \begin{tabular}{cc}
      \toprule
      label & type \\\midrule
      $0$ & $(1,0,0)$\\
      $1$ & $(0,2,1)$\\
      $2$ & $(0,1,3)$\\
      $3$ & $(0,0,6)$\\
      \bottomrule
    \end{tabular}
  \end{minipage}

  \caption{Degeneration graph for quantum $\mathbb{P}^3$'s}
  \label{figure:degeneration-P3}
\end{figure}

\subsection{Quantum $\mathbb{P}^4$'s}
The degeneration graph is given in figure~\ref{figure:degeneration-P4}. There are in total~314 adequate collections, falling into~16~$\Sym_5$-orbits.

This time, the degeneration graph no longer is totally ordered. For example, take the configurations~$4_a$ and~$4_b$. These differ by \emph{how} the two~$\mathbb{P}^2$'s intersect: as we are working in an ambient~$\mathbb{P}^4$ this happens in either a point or a line. Via similar arguments it is possible to describe each of these configurations.

Observe that~$3_a$ and~$3_c$ have the same type, but they are not the same configuration: $3_c$ corresponds to three~$\mathbb{P}^2$'s intersecting in a common~$\mathbb{P}^1$, whereas orbit~$3_a$ has two~$\mathbb{P}^2$'s intersecting only in a point and a third~$\mathbb{P}^2$ intersecting the first in two different~$\mathbb{P}^1$'s.

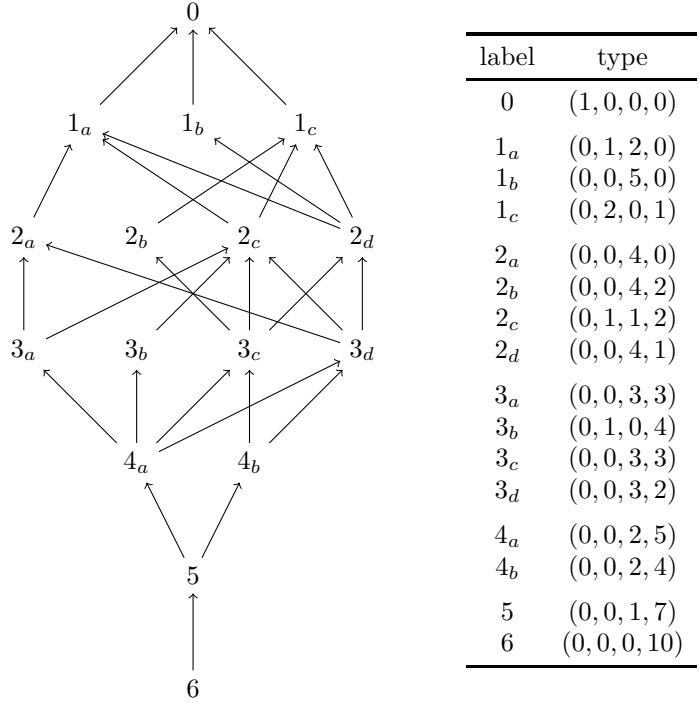
\begin{figure}[ht]
  \begin{minipage}{.4\textwidth}
    \centering
    \begin{tikzpicture}[scale = 1.5]
    
      \node (n11) at (0, 0) {$6$};
    
      \node (n2) at (0, 1) {$5$};
    
      \node (n0)  at (-.5, 2) {$4_a$};
      \node (n15) at (.5, 2)  {$4_b$};
    
      \node (n9)  at (-1.5, 3) {$3_a$};
      \node (n6)  at (-.5, 3)  {$3_b$};
      \node (n1)  at (.5, 3)   {$3_c$};
      \node (n10) at (1.5, 3)  {$3_d$};
    
      \node (n13)  at (-1.5, 4) {$2_a$};
      \node (n12) at (-.5, 4)   {$2_b$};
      \node (n8) at (.5, 4)     {$2_c$};
      \node (n14) at (1.5, 4)   {$2_d$};
    
      \node (n7) at (-1, 5) {$1_a$};
      \node (n4) at (0, 5)  {$1_b$};
      \node (n3) at (1, 5)  {$1_c$};
    
      \node (n5) at (0, 6) {$0$};
    
      \foreach \i/\j in {11/2,
                         2/0, 2/15,
                         0/1, 0/6, 0/9, 0/10, 15/1, 15/10,
                         1/8, 1/12, 1/14, 6/8, 9/8, 9/13, 10/8, 10/13, 10/14,
                         8/3, 8/7, 12/3, 13/7, 14/3, 14/4, 14/7,
                         3/5, 4/5, 7/5}
        \draw[->] (n\i) -- (n\j);
    \end{tikzpicture}
  \end{minipage}
  \begin{minipage}{.4\textwidth}
    \centering
    \begin{tabular}{cc}
      \toprule
      label & type \\\midrule
      $0$ & $(1,0,0,0)$ \\\addlinespace
      $1_a$ & $(0,1,2,0)$ \\
      $1_b$ & $(0,0,5,0)$ \\
      $1_c$ & $(0,2,0,1)$ \\\addlinespace
      $2_a$ & $(0,0,4,0)$ \\
      $2_b$ & $(0,0,4,2)$ \\
      $2_c$ & $(0,1,1,2)$ \\
      $2_d$ & $(0,0,4,1)$ \\\addlinespace
      $3_a$ & $(0,0,3,3)$ \\
      $3_b$ & $(0,1,0,4)$ \\
      $3_c$ & $(0,0,3,3)$ \\
      $3_d$ & $(0,0,3,2)$ \\\addlinespace
      $4_a$ & $(0,0,2,5)$ \\
      $4_b$ & $(0,0,2,4)$ \\\addlinespace
      $5$ & $(0,0,1,7)$ \\
      $6$ & $(0,0,0,10)$ \\
      \bottomrule
    \end{tabular}
  \end{minipage}

  \caption{Degeneration graph for quantum~$\mathbb{P}^4$'s}
  \label{figure:degeneration-P4}
\end{figure}

\subsection{Quantum $\mathbb{P}^5$'s}
For $n=5$, we observe a new phenomenon.
\begin{theorem}
There are at least 2 end points in the degeneration graph for quantum polynomial algebras in 6 variables.
\end{theorem}
\begin{proof}
An endpoint in the graph corresponds to a $n$-dimensional family of quantum polynomial algebras. Let $C$ be the collection 
\[
\{ \mathbb{P}(0,1,2), \mathbb{P}(1,2,3), \mathbb{P}(2,3,4), \mathbb{P}(0,3,4), \mathbb{P}(0,1,4), \]
\[
\mathbb{P}(0,2,5),\mathbb{P}(1,3,5), \mathbb{P}(2,4,5), \mathbb{P}(0,3,5), \mathbb{P}(1,4,5) \}. \]
Then the complement of $C$ is adequate. We have already constructed an algebra $A'$ with exactly the union of these $\PP^2$'s in its point variety. We will show that the family of quantum polynomial algebras with these $\PP^2$ in its point variety is $5$-dimensional. Using the action of $\mathbb{K}$, we may assume that for all~$0\leq i \leq 4$ we have~$q_{i5} = 1$. If we can now show that there are a finite number of solutions, we are done as we have used up all degrees of freedom. It follows from the second row of $\PP^2$'s in the point variety that 
\[
q_{02} = q_{13}=q_{24}=q_{03}=q_{14} = 1.
\]
Using the first four $\PP^2$'s, we get the conditions
\[
q_{01}=q_{23}=a, q_{12}=q_{34} = q_{04} = a^{-1}.
\]
Now, $\mathbb{P}(0,1,4)$ belongs to the point variety if and only if $a = a^{-1}$ or equivalently, $a = \pm 1$. The case $a=1$ leads to the commutative polynomial ring, while $a=-1$ gives an quantum polynomial ring with exactly these 10 $\PP^2$'s in its point variety.
\end{proof}

\end{document}